\documentclass[12pt]{article}
\usepackage{amssymb,amsmath}
\input{amssym.def}\input{amssym}
\usepackage{graphicx, tikz}
\usepackage{mathtools}

\def\subjclass#1{\par\medskip
\noindent\textbf{Mathematics Subject Classification (2010):} #1}
\def\keywords#1{\par\medskip
\noindent\textbf{Keywords.} #1}

\newcommand{\C}{{\mathbb C}}

\newcommand{\bI}{{\mathbb I}}

\newcommand{\N}{{\mathbb N}}

\newcommand{\R}{{\mathbb R}}

\newcommand{\Z}{{\mathbb Z}}
\renewcommand{\S}{{\mathbb S}}

\newcommand\cC{{\mathcal C}}

\newcommand\cQ{{\mathcal Q}}

\newcommand\cS{{\mathcal S}}

\newcommand\bH{{\mathbb H}}

\newtheorem{theorem}{Theorem}[section]
\newtheorem{prop}{Proposition}[section]

\newtheorem{lemma}{Lemma}[section]

\newtheorem{example}{Example}[section]
\newenvironment{proof}{\noindent {\bf Proof.}}{ \hfill $\Box$\\ }

\def\eps{\varepsilon}

\def\dist{\text{\tiny dist}\,}
\def\sgn{\operatorname{\text{sgn}}}
\def\gcd{\mbox{gcd}}

\def\tmax{t_{\max}}
\def\Dtheta{\Delta\theta}

\title{On Sina\u{\i}\ billiards on flat surfaces with horns}
\author{Henk Bruin
\thanks{Faculty of Mathematics, University of Vienna, 
Oskar Morgensternplatz 1, 1090 Vienna, Austria; {\it henk.bruin@univie.ac.at}}
}
\date{\today}

\begin{document}

\maketitle

\abstract{We show that certain billiard flows on planar billiard tables with horns can be modeled as suspension 
flows over Young towers \cite{Young98} with exponential tails. This implies exponential decay of correlations for the billiard 
map.
Because the height function of the suspension flow itself is polynomial when the horns are Torricelli-like trumpets,
one can derive Limit Laws for the billiard flow, including Stable Limits if the parameter of the Torricelli trumpet is
chosen in $(1,2)$.
}


\section{Introduction}

Recent results on the statistical properties of non-uniformly hyperbolic flows in dynamical systems
include polynomial mixing rate, when the flow can be modeled as suspension flow over a Gibbs-Markov map 
or over a Young tower, and the roof function $h$ of the suspension flow has polynomial tails:
\begin{equation}\label{eq:tail}
\mu(\{x \in M : h(x) > t\}) = \ell(t) t^{-\beta},
\end{equation}
where $\ell(t)$ is a slowly varying function.
There are geodesics flows \cite{AD} on non-compact surfaces of curvature $-1$ where
the above tail condition with $\beta = 1$ or $2$ applies (although properties of Kleinian groups
rather than Young towers are used in the modeling).
For Lorentz gas with infinite horizon, the parameter in \eqref{eq:tail} is $\beta = 2$.
So, although the theory puts no restriction on the parameter $\beta$ in \eqref{eq:tail}, these examples provide us only with very specific values of $\beta$.
The model of \cite{CZ05} based on two convex scatterers with points of zero curvature
directly opposite to each other produces finite measure cases with variable $\beta > 1$ for the roof function
of the Young tower (but without giving the slowly varying function in \eqref{eq:tail}).
In dimension one, the Pomeau-Manneville maps allow for inducing schemes where
the induce time satisfies \eqref{eq:tail} for $\beta = 1/\alpha$ where $\alpha$ is the order of contact
between the graph and the tangent at the neutral fixed point.
Thus $\beta > 0$ can be chosen freely, but despite some higher-dimensional variants, Pomeau-Manneville maps
remain too specific to play a substantial role in the modeling of billiards or other mechanical models.
In \cite{Hu00,LM05,HZ16,BT18,EL20} it was shown that almost Anosov diffeomorphisms (and flows \cite{BTT18})
also allow inducing schemes with tails satisfying \eqref{eq:tail}.
These are non-uniformly hyperbolic invertible systems, and, in contrast to Pomeau-Manneville maps,
can be chosen to be $C^\infty$ or real analytic, even if $\beta$ is non-integer.

The purpose of this paper is to provide a class of examples that fit directly in the context of billiard maps, 
and which can be modeled by suspension flows over Young towers with tails as in \eqref{eq:tail}.
The basic ingredient is the geodesic flow on a surface of revolution, which we call {\em horns},
of which the Torricelli trumpet is an example.
New (or at least we are not aware of explicit calculations in the literature)
is a one-parameter family of Torricelli trumpets, which provide tails as in \eqref{eq:tail} where 
the exponent $\beta$ is equal to the parameter of the family.

Let the billiard table $\cQ$ be a flat compact manifold, such as a torus or a rectangle with reflecting boundaries.
We assume that
\begin{itemize}
 \item  there are finitely many circular horns and/or scatterers  $H_i$, $i = 1, \dots, N$, of radius $r_i$
(so of curvature $1/r_i$);
\item their closures are disjoint closures, so the minimal flight time between collisions
$\tau_{\min} > 0$;
\item the horizon is finite, i.e.,  the maximal flight time $\tau_{\max}$ between collisions is finite.
% In particular, there is a uniform upper bound of the number of consecutive tangent (grazing) collisions.
\item Scatterers are ``hard balls'', i.e., the collision rule of the particle with such scatterers is 
the rule of fully elastic reflection. 
\item Horns act as ``soft balls'',  in the sense that they reflect the particle, 
but not according to the law of elastic collision: although the angle of incidence equals the angle of reflection
(up to a minus sign: $\varphi^+ = - \varphi^-$), the entrance position 
on $\partial H_i$ is not necessarily the exit position.
\end{itemize}
A unit mass, unit speed particle moves on this surface with scatterers. It reflects fully elastically at the 
scatterers, but when it meets a horn $H_i$, it moves up on the surface of the horn, 
keeping its speed but observing the law of preservation of angular momentum and the holonomic 
constraint keeping it in $H_i$, until it exits $H_i$ again and resumes its trajectory on $\cQ$.
The excursion time is $2\tmax(\varphi^+)$ where $\varphi^+ \in [-\frac{\pi}{2}, \frac{\pi}{2}]$ 
is the angle of incidence that the particles trajectory makes with the normal vector to $\partial H_i$,
and $\tmax$ is the time for an excursion to reach the highest point in the horn. 
Due to the radial symmetry of $H_i$, the angle of incidence $\varphi^- = - \varphi^+$.

We denote the flow on $\cQ \cup \bigcup_{i=1}^N H_i$ by $\phi^t$.
The excursions of the particle on the horn can take an unbounded amount of time,
so that, despite the bounded distance between, the flow-time between incoming collisions
can be unbounded.

Let us parametrize the circle $\partial H_j$ by the position $\theta^-$ 
(measured clockwise as an angle in $[0,2\pi)$ but in order to avoid confusion 
with the angle of incidence/reflection, we will refer to $\theta^-$ as the position).
The exit position $\theta^+$ is a function of the entrance position $\theta^-$ and the angle of reflection 
$\varphi^+$. The {\em rotation function}\footnote{using the terminology in \cite{BT1}}
$$
\Dtheta := \theta^+ - \theta^-
$$
depends on $\varphi^+$ but (due to radial symmetry) not on $\theta^-$.
We prefer to let $\Dtheta$ depend on the outgoing angle $\varphi^+$,
in order to follow the conditions of B\'alint \& T\'oth \cite{BT1,BT2}, see Section~\ref{sec:condYT}.
Note that $\Dtheta = 0$ at scatterers, and at horns for $\varphi^+=\pm\frac{\pi}{2}$, i.e.,
grazing collisions. Since  $\varphi^+ = - \varphi^-$, the resulting reflection map
\begin{equation}\label{eq:R}
R:M^- \to M^+, 
\qquad (\theta^-,\varphi^-) \mapsto (\theta^+, \varphi^+) = (\theta^- + \Dtheta(-\varphi^-), -\varphi^-)
\end{equation}
represents the outgoing position and angle as function of the incoming position and angle.
Here $M^\pm = \cup_{i=1}^N M^\pm_i$ for
$M^\pm_i = \partial H_i \times [-\frac{\pi}{2}, \frac{\pi}{2}]$
are the incoming and outgoing phase spaces, copies of one another, but formally not the same.

The flight map $F:M^+ \to M^-$ is given by 
\begin{equation}\label{eq:F}
(\theta^+_i,\varphi^+_i) \in M^+_i \mapsto (\theta^-_j,\varphi^-_j) 
= (\theta^-_j, \pi + \varphi^+_j + \theta^-_j - \theta^+_i) \in M^-_j,
\end{equation}
where $i,j$ are the indices of the scatterers or horns of the consecutive collisions.
The composition $T = R \circ F: M^+ \to M^+$ is the billiard map, expressed in outgoing coordinates 
from one horn or scatterer to the next.
Note that $T$ has singularities at $\cup_{i=1}^N \partial H_i \times \{ \pm\frac{\pi}{2} \}$ 
and at $\cup_{\text{\tiny horns } H_i} \partial H_i \times \{ 0 \}$ (non-compactness of horns).

It is worth comparing the collision with horns to scatterers with finite range potentials $V$,
as considered in e.g.\ \cite{Baldwin,D,DL,Knauf1,Knauf2,Kubo}.
Both are modeled by the formula~\eqref{eq:R} where in case of finite range, radially symmetric, potentials $V$ 
and energy level $E$,
$$
\Dtheta = 2\int_{r_{\min}}^{r_i} \frac{ r_i \sin \varphi \ dr}{\sqrt{ r^2(1-2V(r)) - r^2_i \sin^2 \varphi} },
$$
see \cite[Formula (3.7))]{DL} or \cite[Formula (5.2)]{BT1} for the energy level $E = \frac12$ of a unit mass.
An important quantity is the derivative of the rotation function $\Dtheta$:
\begin{equation}\label{eq:kappa0}
\kappa(\varphi) := \frac{d}{d\varphi} \Dtheta(\varphi) \qquad \text{for } \varphi = \varphi^+.
\end{equation}
In \cite{DL} it is shown that the billiard flow is hyperbolic and ergodic if the range of $\kappa$ 
is disjoint\footnote{Since \cite{DL}
measures $\Dtheta$ in incoming angles rather than outgoing as we do, 
their Theorem 4.2 and 4.3 give the forbidden range $[2-\delta,2]$}
from $[-2,-2+\delta]$ for some $\delta > 0$. 
As shown in \cite[Proposition 5.2]{D}, if the scatterer has a smooth finite range potential,
then $\lim_{\varphi \to \pm \frac{\pi}{2}} \kappa(\varphi) = -2$, and indeed, there are several
results showing that the ergodicity of the billiard map can fail, cf.\ \cite{Baldwin,D,Knauf1,Kubo,S63}.
It is not clear, however, that our horns can be modeled as scatterers with finite range potentials.
For instance, sojourn times in
the horn (and hence $\Dtheta$) are not bounded, contrary to what happens in scatterers with
finite range potentials at (all but finitely many) fixed energy levels.

% For lack of a better word, we call our systems billiard flows on a flat table with horns. 
% The horns we choose are Torricelli-like trumpets% and pseudo-spheres,
% covered in Section~\ref{sec:torri}. % and \ref{sec:pseudo}, respectively.
As usual in billiards, $T$ preserves a measure $\mu$ 
that is absolutely continuous to Lebesgue measure,
with density 
$$
d\mu = \frac{r_i}{\Lambda}  \cos \varphi^+ \, d\theta^+ d\varphi^+, \qquad (\theta^+, \varphi^+) \in M^+,
$$
for the normalizing constant $\Lambda = 4\pi\sum_{i=1}^N r_i$.
With respect to this measure, the Sina\u{\i} billiards is known to be ergodic, 
mixing and even Bernoulli, \cite{BS,GO,S}. Our setting is similar enough to conclude
mixing (see \cite[Section 6.7]{CM}), but that doesn't give any quantitative results.

\begin{theorem}\label{thm:YTsusp}
 The billiard flow on billiard tables with horns can be modeled as a suspension flow over a Young tower 
 with exponential tails, see Section~\ref{sec:YT}. The height function $h$ of the suspension
 has polynomial tails $\mu(\{x \in M : h(x) > t\}) \sim C t^{-\beta}$ for some constant $C > 0$  
  if the horns\footnote{The statistical behavior is 
  governed by the trumpet with the smallest parameter.} are Torricelli trumpets 
  with $\beta > 0$ equal to the parameter of the trumpet.
\end{theorem}

It follows that the billiard map has non-zero Lyapunov exponents.
This theorem implies exponential mixing rates for the billiard map, and limit laws of the flow w.r.t.\
the Liouville measure. We state these results later on (Theorem~\ref{thm:map} and \ref{thm:flow}) as they 
can be taken from the literature.
Although soft-ball billiards will be one of our main tools,
the flow on Torricelli trumpets %and pseudo-spheres are 
is a flow on a negatively curved surface (in fact, the curvature tends to zero in the cusp).
There is extensive literature on such flows, e.g.\ \cite{Cou,CS,DP,PW} and references therein to give a sample,
but these are predominantly concerned with ergodicity, mixing and Lyapunov exponents.
The application to limit laws, and the exact computations of our types of horns (despite similarities to 
\cite[Section 2]{PW}) seem to be new.

The next section is concerned with building a Young tower for the billiard map \cite{Young98}, 
or rather verifying that the methods of Chernov \cite{Cher} and the soft-ball billiard approach 
of B\'alint \& T\'oth \cite{BT1,BT2} applies under appropriate conditions.
In Section~\ref{sec:flowhorn} we then estimate the sojourn times on horn, that lead to the height of the 
suspension flow of Theorem~\ref{thm:YTsusp}.
\\[3mm]
\noindent {\bf Notation:}
We will write $(\theta,\varphi) \in M$ for the  position and angle at outgoing collisions, 
and use $(\theta^+,\varphi^+) \in M^+$
only if we want to emphasize that it is about the outgoing collision.
% (and then $(\theta^+,\varphi^+) = R(\theta^-,\varphi^-) = (\theta^-+\Dtheta(-\varphi^-), -\varphi^-)$
% is used for the outgoing collision).
We write $a_n \sim b_n$ if $\lim_n a_n/b_n=1$ and $a_n \approx b_n$ if $a_n/b_n$ 
have a bounded and positive $\limsup$ and $\liminf$.
\\[3mm]
{\bf Acknowledgments:}
HB gratefully acknowledges the support of FWF grant P31950-N45 and Stiftung A\"OU Project 103\"ou6.
He also wants to thank P\'eter B\'alint for his explanations of his and Chernov's papers \cite{BT1,BT2,Cher},
Ian Melbourne for his input on the literature on limit laws,
and Homero Canales for verifying some of the lengthier computations in this paper.
Thanks also to the referees for their vigilance and many helpful remarks.

\section{Billiard maps for tables with horns}

\subsection{Conditions to build a Young tower}\label{sec:condYT}

Young \cite{Young98} introduced a tower construction and used it (among other things) 
to prove that the billiard map of the Sina\u{\i} has exponential decay of correlations.
Chernov \cite{Cher} formulated general conditions under which Young tower 
with exponential or with polynomial tails for various other billiards besides the Sina\u{\i}\ billiard.
One is the uniform hyperbolicity of the billiard map, so (despite its discontinuities) with uniform
expansion and contraction rates, and the angles between stable and unstable leaves
uniformly bounded away from zero.
We discuss this for our setting %in Propositon~\ref{prop:tranversal}  
in Section~\ref{sec:dist}.
Additionally distortion has to be controlled, also in order to find a differentiable quotient map (after 
dividing out the stable direction). Global distortion control is
impossible due to grazing collisions (i.e., collisions with $\varphi = \pm\frac{\pi}{2}$) at scatterers.
Homogeneity strips are therefore introduced in the phase space near $\varphi= \pm\frac{\pi}{2}$
within which distortion control is feasible. This leads, however, to the chopping of unstable leaves
and the need for a ``growth of unstable manifold'' condition in \cite[Section 2]{Cher}. We discuss this for our setting 
in Section~\ref{sec:YT}.

In their turn, B\'alint \& T\'oth give in Definitions 2 and 3 of \cite{BT1} sufficient conditions 
in the soft-ball scatterer setting to apply the methods of Chernov. We rely on
 \cite{BT1} for the verification of \cite[Formulas (26)-(26)]{Cher}.
We summarize these conditions, using their notation, specifically
the derivative of the rotation function $\Dtheta$:
$$
\kappa(\varphi) := \frac{d}{d\varphi} \Dtheta(\varphi) \qquad \text{for } \varphi = \varphi^+.
$$
\begin{enumerate}
 \item $\inf_{\varphi} |2+\kappa(\varphi)| > 0$.
 \item $\tau_{\min} > \sup_{\varphi} -2r_i \frac{\kappa(\varphi)}{\omega(\varphi)}$ where  $r_i$
 is the radius of the scatterers and
 $\omega(\varphi) := \frac{2+\kappa(\varphi)}{\cos \varphi}$.
 \item $\Dtheta$ is piecewise H\"older, i.e., there is $C > 0$ and $\alpha \in (0,1)$ such that
 $$
 |\Dtheta(\varphi) - \Dtheta(\varphi')| \leq C |\varphi-\varphi'|^\alpha
 $$
 for all second coordinates $\varphi, \varphi'$ of points in the same element of a finite partition of the phase space $M^+$.
 \item $\Dtheta$ is piecewise $C^2$ on the interiors of the partition in the previous item.
 \item There is $C > 0$ such that $|\frac{d\kappa(\varphi)}{d\varphi}| \leq C|2+\kappa(\varphi)|^3$.
 \item There is $\eps > 0$ such that $\omega(\varphi)$ from item 2.\ is monotone on one-sided neighborhoods
 $[\varphi^*-\eps,\varphi^*)$ and $(\varphi^*, \varphi^*+\eps]$ of angles $\varphi^*$ 
 where $\kappa$ is infinite or discontinuous.
\end{enumerate}

The importance of these conditions is underlined by the fact that hyperbolicity and een ergodicity can fail
even if $\kappa = -2$, see \cite{TR} where finite range potentials and 
configurations near grazing collisions are established that lead to
homoclinic orbits with nearby elliptic islands.

We list some comments on these properties for Torricelli trumpets, 
and mention where in this paper they are addressed further.

\begin{enumerate}
\item[ad 1.] This condition holds because $\kappa(\varphi) < -2$. % for Torrcelli trumpets.
% fails for the pseudo-spheres in Section~\ref{sec:pseudo}, but only because
% $\kappa = -2$ is possible at grazing collisions ($\varphi = \pm \frac{\pi}{2}$). 
% However, the finite horizon assumption implies that there is an upper bound to the number 
% of consecutive grazing collisions, after which uniform dispersion is guaranteed.
\iffalse
Different types of horns, for example spherical caps, would compromise 
the hyperbolicity of the billiard map, and more directly the dispersive nature of the collisions with horns. 
One solution here would be to insert extra scatterers (or dispersive horns),
positioned at sufficient distance from the horn to make up for their lack of dispersion.
But in this paper, we don't consider this type of horns. %see Section~\ref{sec:wave}. 
\fi
\item[ad 2.] This condition holds because $\omega(\varphi) < 0$ and 
$\sup_{\varphi} -2r_i \frac{\kappa(\varphi)}{\omega(\varphi)} = 0$.
% for Torrcelli trumpets.
% Since $\kappa(\varphi) = -2$ at $\varphi = \pm \frac{\pi}{2}$
% for pseudo-spheres, also this condition fails. 
% However, it is used only to get sufficient dispersion
% of wave-fronts, and our solution in ad 1.\ takes care of this as well.
%Section~\ref{sec:wave} replaces the need for item 2.  
However, if the horns have the shape of pseudo-spheres, which give exponential tails in 
Theorem~\ref{thm:YTsusp}, $\kappa(\pm \frac{\pi}{2}) = -2$ at grazing collisions
and $\sup_{\varphi} -2r_i \frac{\kappa(\varphi)}{\omega(\varphi)} = \infty$.
Hence hyperbolicity and even ergodicity when the horns are pseudo-spheres remain unproven.

\item[ad 3.] H\"older continuity of the rotation function $\Dtheta(\varphi)$ holds in our case for
 $\varphi \approx \pm \frac{\pi}{2}$, 
 but fails near head-on collisions with horns (i.e., $\varphi \approx 0$).
 In fact, if $\varphi = 0$, the particle will never leave the horn again. 
 As we will see, $\Dtheta(\varphi)$ and hence of $\kappa$ are the unbounded.
This situation is not covered in \cite{BT2}; it requires extra arguments 
(in the shape of adding more ``homogeneity strips'') to control the distortion.
% However, the largest expansion in this neighborhood helps in overcoming the effects of
% the extra chopping that this entails. 
More precisely, we will introduce an equivalent of homogeneity strips, denoted by $\bI_{\pm k}$,
which accumulate from both directions on the equator $\{\varphi = 0\}$, within which we 
can control the  distortion of
$\Dtheta(\varphi)$ on these $\bI_k$, see Proposition~\ref{prop:dist}.
Fortunately, unstable leaves become automatically long, in a way that 
the need for additional growth lemmas is avoided.
 
 \item[ad 4.] The rotation function $\Dtheta(\varphi)$ is $C^2$ on all the intervals of continuity
 in $[-\frac{\pi}{2}, \frac{\pi}{2}]$.
 Because $\Dtheta(\varphi)$ blows up near $\{\varphi = 0\}$, we have to resort to $C^2$ smoothness
 on the (artificial) homogeneity strips $\bI_{\pm k}$ near $\{ \varphi = 0\}$. %, but this is unproblematic.
 
 \item[ad 5.] This is unproblematic for Torricelli trumpets
 (see Sections~\ref{sec:torri}).
%  but fails again near grazing collisions at pseudo-sphere,
%  see Section~\ref{sec:pseudo}. However, here both $|\frac{\kappa(\varphi)}{d\varphi}|$ and
%  $|2-\kappa(\varphi)|$ tend to zero as $\varphi \to \pm \frac{\pi}{2}$.
%  This is an irrelevant case for \cite[Formula (3.31), Lemma 4 and Proposition 2]{BT1}.
%  For \cite[Sublemma 1]{BT1} (which is the only other
%  place where item 5.\ is used), it is only used to prove the distortion control of \cite[Proposition 3]{BT1}
%  which we prove by other means in Section~\ref{sec:dist}. 
 %Proposition~\ref{prop:dist} and \eqref{eq:distT}.
 
 \item[ad 6.] This is unproblematic; the computations in Section~\ref{sec:torri}
 yield this condition automatically.
\end{enumerate}

\subsection{Distortion control of the billiard map.}\label{sec:dist}

In this section, we study the distortion of the billiard map $T$. Since the flight map $F$
has bounded distortion inside homogeneity strips (as in \cite[Section 5.3]{CM}, but see Formula~\eqref{eq:distT} below)
and the reflection at scatterers goes as for standard billiard maps, we concentrate on the reflection map $R$
for the horns.
Here the distortion control when $\cos \varphi \approx 0$ is not an issue, as it can be dealt with in the standard way
of introducing homogeneity strips
$$
\bH_{\pm k} = \{ x \in M^+ : |\pm \frac{\pi}{2}-\varphi^+| \in [(k+1)^{-2}, k^{-2}) \},
$$
see \cite[Section 5.3]{CM}, and the fact that $\kappa(\varphi^+)$ is bounded near $\varphi^+ = \pm \frac{\pi}{2}$.
The additional problem occurs for $\varphi^+ \approx 0$
because $\Dtheta$ and hence $\kappa$ are unbounded here.
Our solution is as with grazing collisions at scatterers, introduce homogeneity strips
within which distortion is controlled. Fortunately, the large expansion in such strips overcomes the artificial
chopping in one iterate. Hence the analysis of this case is easier, and doesn't require 
additional growth lemmas.
We  start this section by describing these homogeneity strips, and then deal with the distortion control,
with Proposition~\ref{prop:dist} as main result.

For each horn $H_j$ select an other horn or scatterer $H_i$, opposite to it. 
\begin{itemize}
 \item[(i)]
If $H_i$ is a horn, then there is a maximal arc $A_j \subset M^+_j$ such that for each $x \in A_j$, 
the trajectory starting at $x$ collides with $H_i$ head-on, i.e., $F(x) = (\theta^-_i,0)$,
and such a trajectory will not exit $H_i$ anymore.
\item[(ii)]
If $H_i$ is a scatterer, then $A_j$ 
is the maximal arc $A_j \subset M^+_j$ such that $F(x) \in M^-_i$ and $F \circ T(x) = (\theta^-_j,0) \in M^-_j$.
In other words: trajectories of the flow starting at $x \in A_j$ first bounce with $H_i$ 
before head-on colliding with $H_j$, and thus not exiting $H_j$ again, see Figure~\ref{fig:A}.
\end{itemize}

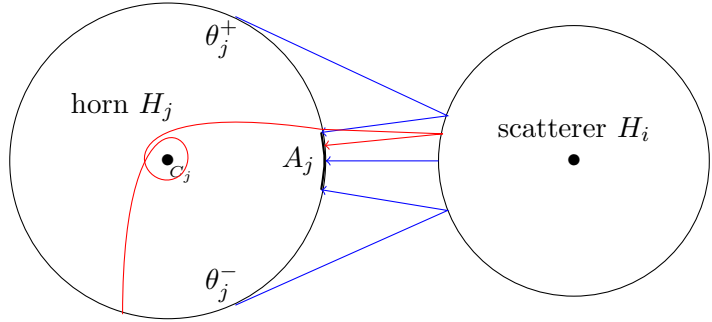
\begin{figure}[ht]
\begin{center}
\begin{tikzpicture}[scale=0.6]
\draw (4,5) circle(3.5); \node at (3,6.2) {\small horn $H_j$};   
   \node at (4,5) {\small $\bullet$}; \node at (4.3,4.7) {\tiny $C_j$};
      \draw (13,5) circle(3); \node at (13,5) {\small $\bullet$}; \node at (13,5.7) {\small scatterer $H_i$};
\draw[->, draw=blue] (10,5) -- (7.5,5);  \node at (6.9,5) {\small $A_j$};    
\draw[->, draw=blue] (5.5,8.2) -- (10.2,6) -- (7.4,5.63);  \node at (5.2,7.7) {\small $\theta^+_j$};
\draw[->, draw=blue] (5.5,1.8) -- (10.2,3.9) -- (7.4,4.36); \node at (5.2,2.3) {\small $\theta^-_j$};
\draw[-, thick] (7.4,5.63) .. controls (7.5,5)  .. (7.4,4.36);
\draw[-, draw=red] (3,1.6) .. controls (3, 7) and (5,5.5)  .. (4.3,4.7);
\draw[-, draw=red] (4.3,4.7) .. controls (3.5,4) and (2,6.5)  .. (7.4,5.7);
\draw[->, draw=red] (7.4,5.7) -- (10.1,5.6) -- (7.47,5.34); 
\end{tikzpicture}
\caption{Trajectories asymptotic to the center $C_j$ of $H_j$.}\label{fig:A}
\end{center}
\end{figure}

In either case, $A_j$ is a smooth curve in $M^+_j$
that stretches across $M^+_j$ in the vertical direction and is transversal to $\{ \varphi^+_j = 0\}$.
Let $b_j = (\theta^+_j, +\frac{\pi}{2})$ and $b'_j = (\theta^+_j, -\frac{\pi}{2})$ be the endpoints of $A_j$.
The reflection map 
$$
R:M^-_j \setminus \{ \varphi^-_j = 0\} \to M^+_j \setminus \{ \varphi^+_j = 0\},
\qquad (\theta^-_j, \varphi^-_j) \mapsto (\theta^- + \Dtheta(-\varphi^-_j), - \varphi^-_j),
$$ 
is a bijection.
The closer to the equator $\S^1 \times \{ 0 \} \subset M^-_j$, the stronger the shear of $R$.
That is, an arc $\{\theta\} \times (0, \frac{\pi}{2}]$ is mapped by $R$ to a spiral curve wrapping infinitely 
often around the annulus $M^+_j$, compactifying on $\S^1 \times \{ 0 \}$ from below,
while $\{\theta\} \times [-\frac{\pi}{2}, 0)$ is mapped by $R$ to a spiral curve wrapping infinitely 
often around the annulus $M^+_j$, compactifying on $\S^1 \times \{ 0 \}$ from above.
Conversely, $\Psi_j := R^{-1}(A_j \setminus \{\varphi^+_j = 0\})$ consists of two spirals
wrapping infinitely often around $M^-_j$ and compactifying on the equator, one from above and one from below.
For every point $(\theta^-_j, \varphi^-_j) \in \Psi_j$, 
\begin{itemize}
 \item[(i)] if $H_i$ is a horn, then $F \circ R(\theta^-_j,\varphi^-_j)$ represents a  particle head-on colliding with $H_i$, so that $T \circ R(\theta^-_j,\varphi^-_j)$ is not defined: the particle never leaves $H_ji$.
 \item[(ii)]  if $H_i$ is a scatterer, $T \circ R(\theta^-_j,\varphi^-_j)$ represents a 
particle outgoing from $H_i$ and head-on colliding with $H_j$, so that $R \circ T \circ R(\theta^-_j,\varphi^-_j)$ 
is not defined: the particle never leaves $H_j$ again.
\end{itemize}
Now $M^-_j \setminus \overline{\Psi_j}$ consists of two strips that wrap around $M^-_j$ infinitely often
and approaching $\{ \varphi = 0\}$ in a spiral fashion from above and below.
Let $e_j = R^{-1}(b_j)$,  $e'_j = R^{-1}(b'_j)$, and let $E_j$ be the straight line
connecting $e_j$ and $e'_j$ in $M^-_j$.
Then $E_j$ cuts $M^-_j \setminus \overline{\Psi_j}$ into infinitely many strips $\bI_{\pm k}$, $k \in \N$,
whose closures are curvilinear rectangles except that they coincide at two opposite corners; note that they wrap
around $M^-_j$ once.

\begin{figure}[ht]
\begin{center}
\begin{tikzpicture}[scale=0.6]
\draw[-] (0,0) -- (6,0) -- (6,4) -- (0,4) -- (0,0);
\draw[.] (0,2) -- (6,2); \node at (-1.2,2) {\small $\varphi = 0$};
\node at (-0.6,0) {\small $-\frac{\pi}{2}$}; \node at (-0.6,4) {\small $+\frac{\pi}{2}$};
\node at (2.5,0) {\small $\bullet$}; \node at (2.4,-0.5) {\small $e_j$};
\node at (3.5,4) {\small $\bullet$}; \node at (3.4,4.6) {\small $e'_j$};
\draw[-, draw=red] (3.5,4) -- (2.5,0); \node at (3.2,1) {\small $E_j$};
\draw[->] (7,2) -- (9,2); \node at (8,2.4) {\small $R$};
\draw[-] (10,0) -- (16,0) -- (16,4) -- (10,4) -- (10,0);
\draw[.] (10,2) -- (16,2); \node at (17,2) {\small $\varphi = 0$};
\node at (16.7,0) {\small $-\frac{\pi}{2}$}; \node at (16.7,4) {\small $+\frac{\pi}{2}$};
\draw[-, draw=blue] (3.5,4) .. controls (4.7, 3.6)  .. (6,3.5);
\draw[-, draw=blue] (0,3.5) .. controls (3, 3.3)  .. (6,3.2);
\draw[-, draw=blue] (0,3.2) -- (6,3.0);
\draw[-, draw=blue] (0,3.0) -- (6,2.73);
\draw[-, draw=blue] (0,2.73) -- (6,2.55);
\draw[-, draw=blue] (0,2.55) -- (6,2.38);
\draw[-, draw=blue] (0,2.38) -- (6,2.2);
\draw[-, draw=blue] (0,2.2) -- (6,2.13);
  \draw[-, draw=blue] (2.5,0) .. controls (1.3, 0.4)  .. (0,0.5);
  \draw[-, draw=blue] (6,0.5) .. controls (3, 0.7)  .. (0,0.8);
  \draw[-, draw=blue] (6,0.8) -- (0,1.0);
\draw[-, draw=blue] (6,1.0) -- (0,1.27);
\draw[-, draw=blue] (6,1.27) -- (0,1.45);
\draw[-, draw=blue] (6,1.45) -- (0,1.62);
\draw[-, draw=blue] (6,1.62) -- (0,1.8);
\draw[-, draw=blue] (6,1.8) -- (0,1.87);
\draw[-, draw=red] (12.5,4) .. controls (11.3, 3.6)  .. (10,3.5);
\draw[-, draw=red] (16,3.5) .. controls (13, 3.3)  .. (10,3.2);
\draw[-, draw=red] (16,3.2) -- (10,3.0);
\draw[-, draw=red] (16,3.0) -- (10,2.73);
\draw[-, draw=red] (16,2.73) -- (10,2.55);
\draw[-, draw=red] (16,2.55) -- (10,2.38);
\draw[-, draw=red] (16,2.38) -- (10,2.2);
\draw[-, draw=red] (16,2.2) -- (10,2.13);
  \draw[-, draw=red] (13.5,0) .. controls (14.7, 0.4)  .. (16,0.5);
  \draw[-, draw=red] (10,0.5) .. controls (13, 0.7)  .. (16,0.8);
  \draw[-, draw=red] (10,0.8) -- (16,1.0);
\draw[-, draw=red] (10,1.0) -- (16,1.27);
\draw[-, draw=red] (10,1.27) -- (16,1.45);
\draw[-, draw=red] (10,1.45) -- (16,1.62);
\draw[-, draw=red] (10,1.62) -- (16,1.8);
\draw[-, draw=red] (10,1.8) -- (16,1.87);
\draw[-, draw=blue] (12.5,4) .. controls (12.7,2) .. (13.5,0); \node at (13.7,1) {\small $A_j$};
\node at (13.5,0) {\small $\bullet$}; \node at (13.4,-0.54) {\small $b'_j$};
\node at (12.5,4) {\small $\bullet$}; \node at (12.4,4.62) {\small $b_j$};
\end{tikzpicture}
\caption{The curves $E_j, \Psi_j = R^{-1}(A_j) \subset M^-_j$ and $A_j, R(E_j) \subset M^+_j$.}\label{fig:Psi}
\end{center}
\end{figure}
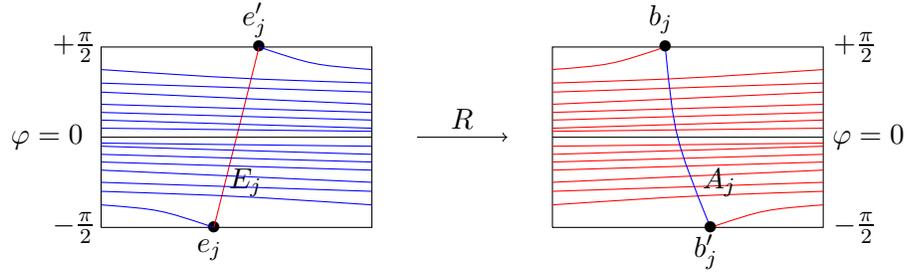

The sets $\bI_{\pm k}$ play the role of homogeneity strips, within which unstable derivatives are uniformly bounded.

\begin{lemma}\label{lem:tildeI}
 Let $\tilde I_{\pm k}$ be the arcs $\bI_{\pm k} \cap (\{ 0 \} \times [-\frac{\pi}{2}, \frac{\pi}{2}])$.
 Then $|\tilde I_{\pm k}| \approx k^{-(1+\frac{1}{\beta})}  = o(d(\tilde I_k, 0))$ as $k \to \infty$
 for some $\beta > 0$, where $d$ is the Euclidean distance on
 $[-\frac{\pi}{2}, \frac{\pi}{2}]$.
\end{lemma}

\begin{proof}
 The precise computation depends on the shape of the horn $H_j$, but there is always a leading term
 of the map $\Dtheta: [-\frac{\pi}{2}, \frac{\pi}{2}] \to  \R$ of the form $g:\varphi \mapsto C \varphi^{-\beta}$
 for some $0 \neq C \in \R$ and $\beta > 0$.
 Now 
 $$
\tilde I_{k} \sim g^{-1}( 2\pi k, 2\pi (k+1)) = 
\left( \left(\frac{C}{2\pi (k+1)}\right)^{1/\beta} \ , \ \left(\frac{C}{2\pi k}\right)^{1/\beta} \right),
 $$
 so $|\tilde I_k| \sim \frac{1}{\beta k} \left(\frac{C}{2\pi k}\right)^{1/\beta}  \sim \frac{1}{\beta k} d(\tilde I_k, 0)$. 
 The same argument works for $-k$.
\end{proof}

\begin{prop}\label{prop:tranversal}
Assume that $\kappa(\varphi^+) \leq 0$ on $M^+$.
Then the stable and unstable leaves $W^{u/s}$ of the billiard map $T$ are uniformly transversal to each other.
\end{prop}

\begin{proof}
In standard hard ball billiards, written in $\varphi_i^+$ and arc-length coordinate $s_i$
parametrizing $\partial H_i$, with associated coordinates $d\varphi^+_i$ and $ds_i$
in the associate tangent space $\mathcal TM_i$, the cones fields
$$
\tilde \cC^u_{(s_i,\varphi^+_i)} := \{ ds_i \cdot d\varphi^+_i \geq 0\} \ \text{ and } \ 
\tilde \cC^s_{(s_i,\varphi^+_i)} := \{ ds_i \cdot d\varphi^+_i \leq 0\} 
$$
serve as unstable and stable cone fields.
Indeed, the derivative of the corresponding billiard map from scatterer $i$ to scatterer $j$ is
$D\tilde T = \begin{pmatrix} 1 & 0 \\ 0 & -1 \end{pmatrix} \cdot D\tilde F$,
where the derivative of the flight map
$$
D\tilde F = \frac{-1}{\cos \varphi^+_j}
\begin{pmatrix}
 \frac{\tau}{r_i} + \cos \varphi^+_i & \tau \\[1mm]
 -\frac{\tau}{r_ir_j} - \frac{\cos \varphi^+_j}{r_i} - \frac{\cos \varphi^+_i}{r_j} \quad & 
 -\frac{\tau}{r_j} - \cos \varphi^+_j
\end{pmatrix},
$$
in coordinates $(s_i,\varphi^+_i)$ and $(s_j,\varphi^+_j)$ can be easily reconstructed from
see \cite[Formula (2.26)]{CM}. Here $\tau$ is the flight time between $H_i$ and $H_j$,
and we wrote $1/r_i$ and $1/r_j$ for their curvatures.
Since all the entries of $D\tilde T$ are negative, the above cones are indeed preserved under forward
and negative times respectively. The difference with our setting is:
\begin{itemize}
 \item We use $\theta_i^+ = r_i s_i$ and $\theta_i^+ = r_i s_i$ as coordinates. The necessary change of coordinates 
 requires multiplying $D\tilde F$ on the left and right with the matrices
 $\begin{pmatrix} 1 & 0 \\ 0 & r_j \end{pmatrix}$ and
$\begin{pmatrix} 1 & 0 \\ 0 & 1/r_i \end{pmatrix}$, respectively. This gives
\begin{equation}\label{eq:DF}
 DF %= \begin{pmatrix} 1 & 0 \\ 0 & r_j \end{pmatrix} \cdot D\tilde F \cdot \begin{pmatrix} 1 & 0 \\ 0 & 1/r_i \end{pmatrix}
 = -\frac{1}{\cos \varphi^+_j} \begin{pmatrix} \frac{\tau}{r_i} + \cos \varphi^+_i  \quad
& \frac{\tau}{r_i}  \\[1mm]
-\frac{\tau}{r_i} - \frac{r_j \cos \varphi^+_j}{r_i} - \cos \varphi^+_i
& -\frac{\tau}{r_i} - \frac{r_j \cos \varphi^+_j}{r_i} \end{pmatrix}.
\end{equation}

\item The extra shear of the reflection map $R:M^-_j \to M^+_j$, 
$(\theta^-_j,\varphi^-_j)\mapsto(\theta^- + \Dtheta, -\varphi^-_j)$, with derivative
 $\begin{pmatrix} 1 & \kappa \\ 0 & -1 \end{pmatrix}$ with $\kappa = \frac{d\Dtheta(\varphi^+)}{d\varphi^+}$, 
 instead of the derivative
  $\begin{pmatrix} 1 & 0 \\ 0 & -1 \end{pmatrix}$ of the reflection map of a standard scatterer.
\end{itemize}
The resulting derivative is
\begin{eqnarray*}
DT &=& \begin{pmatrix} 1 & \kappa \\ 0 & -1 \end{pmatrix} \cdot DF \\
&=& \frac{-1}{\cos \varphi^+_j} 
\begin{pmatrix} \frac{\tau}{r_i} + \cos \varphi^+_i - \frac{\kappa \tau}{r_i}
- \frac{\kappa r_j \cos \varphi^+_j}{r_i} - \kappa \cos \varphi^+_i \quad
& \frac{\tau}{r_i} - \frac{\kappa \tau r}{r_i} - \frac{\kappa r_j \cos \varphi_j^+}{r_i} \\[1mm]
\frac{\tau}{r_i} + \frac{r_j \cos \varphi_j}{r_i} + \cos \varphi^+_i
& \frac{\tau}{r_i} + \frac{r_j \cos \varphi^+_j}{r_i} \end{pmatrix}.
\end{eqnarray*}
Since $\kappa \leq 0$, again all entries of $DT$ are negative, so
\begin{equation}\label{eq:cones}
 \cC^u_{(\theta^+,\varphi^+)} := \{ d\theta^+ \cdot d\varphi^+ \geq 0\} \ \text{ and } \ 
 \cC^s_{(\theta^+,\varphi^+)} := \{ d\theta^+ \cdot d\varphi^+ \leq 0\} 
\end{equation}
are preserved under forward and backward iteration of $DT$, respectively.
If $\kappa$ is very negative, then $DT_{(\theta^+, \varphi^+)}(\cC^u_{(\theta^+,\varphi^+)})$ 
aligns itself with the horizontal axis.
However, the images of the stable cones $DT^{-1}(C^s_{T(\theta^+,\varphi^+)})$ are uniformly 
compactly contained in $C^s_{(\theta^+,\varphi^+)}$ because the matrix $DF^{-1}$ is applied {\bf after}\footnote{If 
we had used a billiard map $T:M^- \to M^-$ in terms of incoming coordinates, it would have been the unstable
image cones rather than the stable image cones, that would have been uniformly compactly contained
in the unstable cone fields.}
(the large shear of) $DR^{-1}$. Hence the angle between  stable and unstable leaves is uniformly bounded 
away from zero.
This proves the lemma.
\end{proof}

The backward singularity sets $\cS^{m} := \cup_{j=0}^m T^{-j}(\bigcup_i \partial H_i \times 
\{-\frac{\pi}{2},\frac{\pi}{2} \}\, \cup\, \bigcup_{\text{\tiny \em horns }H_j} \partial H_j \times \{ 0 \})$ for  
$m \geq 1$, also belong to the stable cone field, and the
forward singularity sets 
$\cS^{-m} := \cup_{j=0}^m T^{j}((\bigcup_i \partial H_i \times \{-\frac{\pi}{2},\frac{\pi}{2} \} \cup \bigcup_{\text{\tiny \em horns }H_j}
\partial H_j \times \{ 0 \})$ for $m \geq 1$, 
belong to the unstable cone field.

\begin{prop}\label{prop:dist}
 Let $W = W^u$ be an unstable leaf contained in $\bI_{\pm k}$ and bounded away from $\{ \varphi = \pm \frac{\pi}{2}\}$.
 Then there is $C_{\dist} \in \R$ such that the distortion in the unstable direction
 \begin{equation}\label{eq:distR}
 \log \frac{ |J^uR(y)| }{ |J^u R(x)| } \leq C_{\dist} d_{R(W)}(R(x), R(y)),
 \end{equation}
 where $d_{R(W)}$ indicates arc-length in $R(W)$.
 In fact, $\log \frac{ |J^uR(y)| }{ |J^uR(x)| } = o( d_{R(W)}(R(x), R(y)))$ as $k \to\infty$, i.e., as $\varphi \to 0$.
\end{prop}

\begin{proof}
 First we assume that $W$ is transversal to horizontal lines in $M$, i.e., transversal to lines of constant
 $\varphi$.
 Thus $W$ can be written as the graph of a $C^1$-function $w:\tilde I_{\pm k} \to \S^1$.
 The length-element of arc-length along $W$ is
 $ds = ds(\varphi) = \sqrt{1+(w'(\varphi)^2)} d\varphi$.
 Recall that $R(\theta,\varphi) = (\theta+\Dtheta(\varphi), -\varphi)$ is the reflection map, and
 $\Dtheta'(\varphi) = \kappa(\varphi)$.
 Then the image of $ds$ under $R$ is
 $$
 dR(s) = \sqrt{ d\varphi^2 + (w'(\varphi)+\kappa(\varphi))^2\ d\varphi^2}
 = \sqrt{1+ (w'(\varphi)+\kappa(\varphi))^2}\ d\varphi.
 $$
 Transversality of $W$ means that $|w'|$ is bounded, so
 $$
 \frac{dR(\varphi)}{d\varphi} = \sqrt{1+ (w'(\varphi)+\kappa(\varphi))^2} \approx \kappa(\varphi) \approx 
 C\beta \varphi^{-(\beta+1)}.
 $$
 where we used that leading term of $\kappa$ is $C\beta \varphi^{-(\beta+1)}$
 as in Lemma~\ref{lem:tildeI}.
 Hence, taking $x = (w(\varphi), \varphi)$ and $y = (w(\varphi+\eps), \varphi+\eps)$, the distortion
 $$
 \log \frac{|J^uR(y)|}{|J^uR(x)|} \approx -(\beta+1) \log(1+\frac{\eps}{\varphi} )
 \approx \frac{-\eps(1+\beta)}{\varphi}.
 $$
 Now to find $d_{R(W)}(R(y),R(x))$ we integrate $dR(s)$ over $[\varphi, \varphi+\eps]$.
 This gives
 \begin{eqnarray*}
 \int_{\varphi}^{\varphi+\eps}  \sqrt{1+ (w'(v)+\kappa(v))^2}\ dv
 &\approx& \int_{\varphi}^{\varphi+\eps} |\kappa(v)|\ dv\\
 &=& |\Dtheta(\varphi+\eps) - \Dtheta(\varphi)| 
 \approx C\left( \varphi^{-\beta} - (\varphi + \eps)^{-\beta} \right) \\
 &=& C \varphi^{-\beta} \left( 1-(1+\frac{\eps}{\varphi})^{-\beta} \right)
 \sim \beta C \eps \varphi^{-(1+\beta)},
\end{eqnarray*}
 where in the last approximation we used that $|\eps|$ is small compared to $|\varphi|$
 as shown in Lemma~\ref{lem:tildeI}. 
Formula \eqref{eq:distR} follows.

For $\varphi \to 0$, we get $1+\sup|w'(\varphi)|^2 = o(\kappa(\varphi))$, so that the $\approx$ 
becomes $\sim$ in this case, and the extra factor $\varphi^{-\beta}$ accounts for the little $o$. 
\end{proof}

The flight map $F$ has similar distortion properties as \eqref{eq:distR}, as can be derived from
the distortion result in \cite[Lemma 5.27]{CM}, except that the exponent on the right-hand side 
is $1/(1+1/\beta)$, where $1+1/\beta$ is the exponent in the width in the strips $|\bI_{\pm k}|
\approx k^{1+1/\beta}$, 
see \cite[Formulas (5.8), (5.21) and Lemma 5.27]{CM} and the adaptation for general exponents in e.g.\ 
\cite[Lemma 3.1]{DZ14}.
Because of Lemma~\ref{lem:tildeI} for $\bI_{\pm k}$ and the usual width of the strips $\bH_{\pm k}$, 
this exponent becomes $\min\{\frac13, \frac{\beta}{\beta+1}\}$.

Because the uniform expansion of the billiard map $T = R \circ F$, we get the following distortion
estimate for the billiard map:
 \begin{equation}\label{eq:distT}
 \log \frac{ |J^uT^n(y)| }{ |J^uT^n(x)| } \leq C_{\dist} 
 d_{T^n(W)}(T^n(x), T^n(y))^{\min\{\frac13, \frac{\beta}{\beta+1}\}},
 \end{equation}
for a potentially large uniform constant $C_{\dist}$ and all $n \geq 1$ and $x,y$ in the same
unstable leaf of $T^n$.
The absolute continuity of holonomies in \cite[Formula (2.3)]{Cher} is a corollary of \eqref{eq:distT},
see \cite{BSC,Young99}.
Together with \cite[Theorem 5.2 and Section 5.7]{CM} this give sufficient distortion control to 
conclude that the quotient tower map $T_{\bar \Delta}$ presented in the next section,
has a H\"older derivative.

\subsection{Building a Young tower with exponential tails}\label{sec:YT} 

A Young tower  \cite{Young98} is a schematic dynamical system, in fact an extension over
a dynamical system $(X, T)$, of the form
$(\Delta, T_{\Delta}, \mu_{\Delta})$, where the space
$$
\Delta = \bigsqcup_i \bigsqcup_{\ell=0}^{\sigma_i-1} \Delta_{i,\ell}.
$$
The sets $\Delta_{i,\ell}$ are copies of the $\Delta_{i,0}$ and the tower map $T_{\Delta}$ acts as
$$
T_{\Delta}: x \in \Delta_{i,\ell} \mapsto \begin{cases}
                                           x \in \Delta_{i,\ell+1} & \text{ if } 0 \leq \ell < \sigma_i-1;\\
                                           T^{\sigma(x)}(x) \in \Delta_0 := \sqcup_i \Delta_{i,0} & \text{ if } \ell = \sigma_i-1,
                                          \end{cases}
$$
and $(\Delta, T_{\Delta})$ factors over $(X,T)$ via $\pi:\Delta \to X$, $\pi(u,\ell) = T^{\ell}(u)$
for $(u,\ell)\in \Delta_{i,\ell}))$.
The return map $T^{\sigma_i}: \Delta_0 \to \Delta_0$ to the base
$\Delta_0 = \sqcup \Delta_{i,0}$ is a uniformly hyperbolic map with 
certain distortion properties, preserving an SRB-measure $\mu_0$.
Here $\sigma:\Delta_0 \to \N$ with $\sigma_i := \sigma|_{\Delta_{i,0}}$ constant for all $i$ 
is called the roof function.
We speak of exponential tails if there is $\lambda \in (0,1)$ such that 
$\mu_0(\{ x : \sigma(x) > n\}) = O(\lambda^n)$.
We can extend $\mu_0$ to an $T_{\Delta}$-invariant measure by setting
$\mu_{\Delta}|_{\Delta_{i,\ell}}  = \bar\sigma^{-1} \mu_0|_{\Delta_{i,0}}$ for normalizing constant
$\bar\sigma = \sum_{n \geq 1} n\mu_0(\{ \sigma(x) = n\})$.
This measure $\mu_{\Delta}$ pushes down to a $T$-invariant SRB-measure on $(X,T)$ via $\mu = \mu_{\Delta} \circ \pi^{-1}$.
The existence of a Young tower with exponential tails implies that the underlying system
$(X,T,\mu)$ is exponentially mixing (provided $\gcd(\sigma_i : i \in \N\} = 1$) and satisfies the Central Limit Theorem 
for H\"older observables, see \cite{Young98}. 
A step in the argument is to consider the quotient tower $(\bar \Delta, T_{\bar \Delta})$ obtained
by collapsing stable leaves to points. The smoothness (Gibbs-Markov) of the quotient map $T^{\sigma_i}_{\bar \Delta}$,
as described in e.g.\ \cite[Section 3.1]{Young98} and \cite[Theorem 5.2 and Section 5.7]{CM} 
relies on the distortion estimates given in Section~\ref{sec:dist}, specifically Formula~\eqref{eq:distT}.

Chernov \cite[Theorem 2.1]{Cher} proved a general theorem on the existence of a Young tower 
with exponential tails for non-uniformly hyperbolic invertible maps, based on a set of conditions
concerning expansion and distortion control along unstable leaves
and specific ``growth of unstable manifolds'' conditions (2.6)-(2.8) in \cite{Cher}.
He continues to verify these conditions for various billiard systems, of which 
the standard Sina\u{\i}\ billiard maps\footnote{i.e., disjoint strictly convex fully elastic scatterers with $C^3$ boundaries
on a compact flat table, first fully treated in \cite{Young98}} is the most relevant to us, 
see \cite[Sections 6 \& 7]{Cher}.
B\'alint \& T\'oth verify these conditions for soft scatterers, expressed as Definition 2 \& 3 in \cite{BT2}.
In the previous sections we verified most of the Chernov resp.\ B\'alint \& T\'oth conditions, and here
we combine these steps to the final verification. That is, we indicate which adaptations in the arguments 
of \cite[Section 7]{Cher} are still required.

Chernov \cite[Section 7]{Cher} uses two metrics to obtain hyperbolic expansion:
\begin{itemize}
\item The $p$-(pseudo-)metric which has the best expansion properties, but only that  after
a close-to-grazing collision with corresponding cut into homogeneity strips, the expansion has a one iterate delay.

\item The Euclidean metric. Now the expansion factor in unstable directions occurs instantaneously at collisions,
but it is not always $\geq 1$.
Therefore a particular iterate $T^m$ of the billiard map $T$ is chosen, which multiplies the number of
discontinuity curves $\cS^{m-1}$ and
$\cup_{n=0}^{m-1} T^{-n}(\cup_j \partial H_j \times \{ 0\})$,
and corresponding boundaries of homogeneity strips
$\cup_{k \geq k_0} \cup_{n=0}^{m-1} T^{-n}(\partial \bI_{\pm k})$.
\end{itemize}
However, combining the two metrics, one can prove uniform expansion (contraction) of
unstable (stable) leaves, see \cite[Lemma 7.1]{Cher}.

Let $W$ be any unstable leave of length $\leq \delta_0$. It may be cut into at most $K_m+1$ pieces
by $\cS^{m-1}$, where $K_m$ depends only on $m$ and the number of scatterers and horns.
In the next $m$ iterate, it may be cut again, even into countably many pieces, by curves in
$\cup_{n=0}^{m-1} T^{-n}(\{ \varphi = 0 \text{ at horns}\} \cup \bigcup_{k \geq k_0} \partial \bI_{\pm k}) \cup
\{ \varphi = \pm \frac{\pi}{2}\} \cup \bigcup_{k \geq k_0} \partial \bH_{\pm k})$.
We label these pieces as
$W_{k_1, \dots, k_m, j}$, where $1 \leq j \leq K_m+1$, $k_i \in \Z$ 
and $T^{m-n}(W_{k_1, \dots, k_m, j}) \subset \bI_{k_n}$. 
Bear in mind that some of these labels can refer to the empty set.
Head-on collisions with horns and grazing collisions have their own homogeneity strips 
$\bI_{\pm k}$ and $\bH_{\pm k}$ where the expansion of the billiard
map is $\approx k^{1+\beta}$ and $\approx k^2$ respectively;
we will use $\nu := \min\{ 2, 1+\beta\}$ for the worst case of the two.
The unstable expansion for $T_1 = T^m$ on a piece $W_{k_1, \dots, k_m}$ of unstable manifold thus becomes
$$
|J^u_1(x)| \geq L_{k_1, \dots, k_m} := \max\{ \Lambda_1, \frac{1}{C_{\text{\tiny exp}\,}}\prod_{k_i \neq 0} k_i^{\nu} \}.
$$
This product $\prod_{k_i \neq 0}$ then reappears in the definition\footnote{This $\Theta$ is called $\theta_0$ in \cite{Cher}.}
of $\Theta := 2\sum_{k \geq k_0} k^{-\nu} \leq \frac{7}{\nu}\, k_0^{1-\nu}$.
We need to choose $k_0$ so large that, as in \cite[Formula (7.5)]{Cher} with corresponding constant $B_0$,
\begin{equation}\label{eq:theta0}
(K_m+1) (\Lambda_1^{-1} + 2B_0 \Theta) < 1.
\end{equation}
Also \cite[Lemma 7.2]{Cher} needs to be adjusted to:

\begin{lemma}\label{lem:7.2}
 For all $\delta > 0$, there is $B = B(m)$ such that
 $$
 \sum_{k_1, \dots, k_m \geq 2} \min\{ \delta, (\kappa_1 \cdots k_m)^{-\nu} \} < B(m) \delta^{\frac{\nu-1}{2m}}.
 $$
\end{lemma}
But the proof goes as in \cite[Appendix]{Cher}, with some minor and obvious adaptations.

Thus we can apply Chernov's main theorem for the billiard map, which we restate here:

\begin{theorem}\label{thm:map}
 For any type of horn discussed in this paper, for every $\alpha \in (0,1)$
 there is $\lambda \in (0,1)$
 such that the billiard map $(M,T)$ has exponential decay of correlations:
 $$
 \left| \int_M v \cdot w \circ T^n \, d\mu - \int_M v \, d\mu \int_M w \, d\mu \right| =  O(\lambda^n)
 $$
 for the SRB-measure $\mu$ and $\alpha$-H\"older functions $v,w:M \to \R$ and also the Central Limit Theorem holds
 for $v$ provided it is not cohomologous to a constant function.
\end{theorem}

\section{The billiard flow}

The billiard flow can now be modeled as a suspension flow over this Young tower, i.e.,
the space is now $\Delta^h := \sqcup_{i,\ell} \Delta_{i,\ell} \times [0,h(x)] / \sim$
where $(x, h(x)) \sim (T_{\Delta}(x),0)$, and the flow
$\phi^t_{\Delta}(x,u) = (x,u+t) \in \Delta^h$.
The height function $h$ is either equal to the (bounded) flight time $\tau(x)$ between a horn/scatterer
and a scatterer, 
or equal to the flight time $\tau(x)$ between a scatterer and a horn plus the sojourn time $2\tmax$
in the horn.
The $\phi^t_{\Delta}$-invariant measure $\mu^h_{\Delta} = \bar{h}^{-1} \mu_\Delta \otimes \mbox{Leb}$
for the normalizing constant $\bar h = \int_\Delta h(x) \, d\mu_\Delta$ or $\bar h = 1$ if
this integral is infinite, because in this infinite measure case, there is no normalization.
The corresponding flow-invariant measure $\mu^h$ is the push-down $\mu_\Delta^h \circ \pi_h^{-1}$
where $\pi_h(x,u) = \phi^u \circ \pi(x)$.

The computations in Section~\ref{sec:torri} show that the tails of $h$ have the asymptotics  
\begin{equation}\label{eq:tail2}
\mu(\{ x \in M^+_i : h \circ F(x) > t\}) = \frac{r_i}{\Lambda}
\int_{\{ x \in M^+_i :  \tau(x) + 2\tmax \circ F(x) > t\}} \cos \varphi \, d\mu \sim Ct^{-\beta}
\end{equation}
for some constant depending only on the shape of the horns. 
In fact, the exponent $\beta$ is equal to the parameter $\beta$ of the Torricelli trumpet,
and therefore $\mu^h$ is finite if and only if $\beta > 1$.

Due to Theorems~\ref{thm:YTsusp} and \ref{thm:map} with \eqref{eq:tail2} 
we can apply results from \cite{MV19} or \cite{BTT18} to derive the following distributional limit theorems for the flow.

\begin{theorem}\label{thm:flow}
Consider a Sina\u{\i} billiard with Torricelli trumpets 
as horns, where $\beta > 1$ is the smallest parameter of these trumpets.
Let $v$ be a measurable bounded\footnote{Note that the space is not compact, so H\"older continuity doesn't imply 
boundedness.}
H\"older observable such that $\inf_{x \in H_i}|v(x)| > 0$ for at least one horn $H_i$ with parameter $\beta$. Then:
\begin{itemize}
\item If $\beta \in (1,2)$, then $v$ satisfies a Stable Law:
 $$
 \frac{1}{T^{1/\beta}} \left(\int_0^T v \circ \phi^t \, dt - T \int v \, d\mu^h \right) 
 \Rightarrow^d G_\beta
 \qquad \text{ as }\ T \to \infty.
 $$ 
\item  If $\beta = 2$, then $v$ satisfies a non-Gaussian Central Limit Theorem:
 $$
 \frac{1}{\sqrt{T \log T}} \left(\int_0^T v \circ \phi^t \, dt - T \int v \, d\mu^h \right) 
 \Rightarrow^d {\mathcal N}(0,1)
 \qquad \text{ as }\ T \to \infty.
 $$
\item If either $\beta > 1$ and $\mbox{supp}(v)$ is compact (rather than containing a horn),
or $\beta > 2$, then $v$ satisfies a standard Central Limit Theorem:
 provided $v$ is not cohomologous to a constant function, there is a constant $\sigma > 0$ such that
 $$
 \frac{1}{\sigma \sqrt{T}} \left(\int_0^T v \circ \phi^t \, dt - T \int v \, d\mu^h \right) \Rightarrow^d {\mathcal N}(0,1)
 \qquad \text{ as }\ T \to \infty.
 $$  
\end{itemize}
\end{theorem}

\subsection{Dynamics of the flow on horns}\label{sec:flowhorn}

Let $H$ be a surface of revolution in $\R^3$ obtained by revolving the curve $x = x(z)$ around the $z$-axis.
We will use the radius $r = r(z) = \sqrt{x^2+y^2}$ as radius of $H$ and $z = z(r)$ is the inverse function.
Thus $H$ has the parametrization
\begin{equation}\label{eq:para}
\sigma(z,\theta) = (r(z) \cos \theta, r(z) \sin \theta, z), \qquad z \geq z_0, \theta \in [0,2\pi).
\end{equation}
Abbreviate $r_0 := r(z_0)$.

\begin{example}
The area and volume of $H$ are 
$$
A := 2\pi \int_{z_0}^\infty (1+r'^2) r \, dz \qquad \text{ and } \qquad 
V := \pi^2 \int_{z_0}^\infty r^2 \, dz
$$
respectively.
For $r(z) = z^{-\beta}$, we get 
$A = \infty$, $V = \frac{\pi}{2\beta-1} < \infty$
(painter's paradox) for $\beta \in (\frac12,1]$.
This holds specifically for the case $\beta = 1$, i.e., $r = 1/z$, when $H$ is called the trumpet of 
Torricelli (or also Gabriel's horn).
For $\beta > 1$ we have $A, V < \infty$ and for $\beta \leq \frac12$, both $A, V = \infty$.
The Gaussian curvature of such a surface 
$$
 \kappa_G = -\frac{r''}{r(1+r'^2)^2} 
 = -\frac{\beta^2+\beta}{z^2 (1+\beta^2 z^{-2(1+\beta)})^2} \to 0 \qquad \text{ as } z \to \infty.
$$
\end{example}

For the next exposition, see \cite[Section 4C]{Arnold}.
A geodesic $\Gamma$ is the path on $H$ traced out by a unit mass particle moving along $H$ at unit speed
with no external forces other than the holonomic constraints keeping it on $H$.
Assume the geodesic starts at $(\theta_0, z_0) \in \partial H$, making an incoming angle
$\varphi_0 \in [-\frac{\pi}{2}, \frac{\pi}{2}]$
with the inward vertical meridian. Since $\varphi^{\pm}$ are angles with the outward normal vector to $\partial H$,
we have $\varphi_0 = - \varphi^- = \varphi^+$, so that $\kappa(\varphi_0) = \kappa(\varphi^+)$ as used in
Section~\ref{sec:condYT}.
The kinetic energy
\begin{equation}\label{eq:Ekin}
E_{kin} = \frac12 |v|^2 = \frac12((1+r'^2) \dot z^2 + r^2 \dot \theta^2) = \frac12
\end{equation}
is one constant of motion.
Due to the rotational symmetry (using Noether's Theorem), the $z$-component of the angular momentum
\begin{equation}\label{eq:angmom}
r^2 \dot \theta = r_0 |v| \sin \varphi_0 = r_0 \sin \varphi_0
\end{equation}
is the second constant of motion.
Inserting $r \dot \theta = |v| \sin \varphi$ (where $\varphi = \varphi(z)$ is the angle with the inward vertical meridian) 
we get a derived constant of motion (Clairaut's Theorem) 
\begin{equation}\label{eq:Clairaut}
r(z) \sin \varphi(z) = r_0 \sin \varphi_0.
\end{equation}
The absolute value of $\sin \varphi$ is largest at the highest point of the geodesic (where
$\sin \varphi_{\max} = 1$, $r_{\min} = r_0 \sin \varphi_0$, $\theta = \theta_{\max}$ and  $z = z_{\max} = z(r_{\min})$), 
then the geodesic spirals down again (symmetrically to the upwards spiral), 
until it hits $\partial H$ at an angle $-\varphi_0$ with the inward vertical meridian.

The question we pose ourselves is:
\begin{quote}
 What is the time $\tmax$ needed of the geodesic particle to reach the top at $z_{\max}$?
\end{quote}
This has a direct consequence for the tails of geodesic flow if these sojourns inside the horns
are modeled by suspension flow with height function $2 \tmax$ and base map
$R: M^-_j \to M^+_j$ as in \eqref{eq:R}.

\subsection{Computation of $\tmax$ and $\theta_{\max}$}

From \eqref{eq:Ekin} combined with \eqref{eq:angmom} we find
$$
\dot z = \frac{dz}{dt} = \sqrt{ \frac{1-r(z)^2 \dot \theta^2}{1+r'(z)^2} }
= \sqrt{ \frac{1-r_0^2r(z)^{-2} \sin^2 \varphi_0}{1+r'(z)^2} }.
$$
Therefore 
$$
dt = \sqrt{ \frac{1+r'(z)^2 }{1-r_0^2r(z)^{-2} \sin^2 \varphi_0} } \, dz
$$
and
$$
\tmax = \int_0^{\tmax} dt = \int_{z_0}^{z_{\max}} \sqrt{ \frac{1+r'(z)^2}{1-r_0^2r(z)^{-2} \sin^2 \varphi_0} } \, dz.
$$
Using the change of coordinates $u = \frac{r_0}{r(z)} |\sin \varphi_0|$, so
$z(r(u)) = z(\frac{r_0|\sin \varphi_0|}{u})$, $z=z_0 \Leftrightarrow u = |\sin \varphi_0|$,
$z= z_{\max} \Leftrightarrow u = 1$ and
$dz =  - \frac{r_0|\sin \varphi_0|}{u^2} \frac{1}{r'(z(\frac{r_0|\sin \varphi_0|}{u}))}\, du$, 
we find\footnote{Because $r' < 0$, we obtain an extra minus sign when moving $r'$ into the square-root.}
\begin{equation}\label{eq:Tmax}
\tmax = r_0|\sin \varphi_0| 
\int_{|\sin \varphi_0|}^1  \frac{1}{u^2} \cdot 
\sqrt{ \frac{1+r'(z(\frac{r_0|\sin \varphi_0|}{u} ))^{-2}}{1-u^2}   } \, du.
\end{equation}
Now for the difference between the entrance position $\theta_0$ and the position $\theta_{\max}$ 
reached at the top of the geodesic, we find, using \eqref{eq:angmom} and the previous computation for $\dot z$:
\begin{eqnarray*}
 \theta_{\max} - \theta_0 &=& \int_{\theta_0}^{\theta_{\max}} d\theta = \int_0^{\tmax} \dot \theta\, dt =
 \int_0^{\tmax} \frac{r_0 \sin \varphi_0}{r^2} \, dt \\
 &=& \int_{z_0}^{z_{\max}} \frac{r_0 \sin \varphi_0}{r^2} \frac{1}{\dot z}\, dz \\
 &=& r_0 \sin \varphi_0 \int_{z_0}^{z_{\max}} \frac{1}{r^2} \sqrt{\frac{1+r'(z)^2}{1-r_0^2 r(z)^{-2} \sin^2\varphi_0} } \, dz.
\end{eqnarray*}
This should be compared to Formula (5.2) in \cite{BT1} expressing $\Dtheta$ in terms of the potential of a soft scatterer. 
Applying the transformation $u = \frac{r_0}{r(z)} |\sin \varphi_0|$ as before, we get
\begin{equation}\label{eq:disp}
\theta_{\max} - \theta_0 = \sgn(\varphi_0) \int_{|\sin \varphi_0|}^1  
\sqrt{ \frac{ 1+r'(z(\frac{r_0|\sin \varphi_0|}{u}))^{-2}}{1-u^2} } \, du.
\end{equation}
Throughout (and following the notation of \cite{BT1,BT2}) we let
\begin{equation}\label{eq:DTh} 
\Dtheta = 2(\theta_{\max} - \theta_0)
 = 2\sgn(\varphi_0)\int_{|\sin \varphi_0|}^1 \sqrt{ \frac{ 1+r'(z(\frac{r_0|\sin \varphi_0|}{u}))^{-2}}{1-u^2} } \, du.
\end{equation}
be the difference in incoming and outgoing angle of the obstacle as function of angle of incidence $\varphi_0$.
Its derivative w.r.t.\ $\varphi_0$ is 
\begin{eqnarray}\label{eq:kappa}
\kappa(\varphi_0) &=& \frac{\partial \Dtheta(\varphi_0)}{\partial \varphi_0}
= - 2 \sqrt{1+(r'(z_0))^{-2}}\ +\ 2r_0 \cos\varphi_0  \times  \nonumber \\
&& \int_{|\sin \varphi_0|}^1 \frac{ z'(\frac{r_0|\sin \varphi_0|}{u})} 
{\sqrt{1+r'(z(\frac{r_0|\sin \varphi_0|}{u}))^2 } } \ 
\frac{r''(z(\frac{r_0|\sin \varphi_0|}{u}))}{r'(z(\frac{r_0|\sin \varphi_0|}{u}))^2}\frac{1}{u\sqrt{1-u^2}}
\, du.
\end{eqnarray}
For convex obstacles, i.e., with $z' < 0$ and $r'' > 0$, the two terms in this expression have the same sign.
The first term $< -2$, whereas the second varies between $0$ (as $\varphi_0 \to \pm \pi/2$) and
potentially $-\infty$ (as $\varphi_0 \to 0^{\mp}$). 
Therefore, unless $r'(z_0) = -\infty$ (as would be the case for a pseudo-sphere) 
$\kappa(\varphi_0)$ is bounded way from $[-2,0]$ as required in \cite{BT2} to obtain uniform hyperbolicity.

\subsection{Torricelli's trumpets}\label{sec:torri}
Assume that the horn is the surface of revolution of the curve $r(z) = z^{-\beta}$,
with $r'(z) = -\beta z^{-(1+\beta)}$,  $r''(z) = \beta(1+\beta) z^{-(2+\beta)}$
and $z(r) = r^{-1/\beta}$.
Inserting the equations for $r(z)$ into \eqref{eq:Tmax} gives
$r'(\frac{r_0 |\sin \varphi_0|}{u}) = \beta^{-2} z(\frac{r_0 |\sin \varphi_0|}{u})^{2(1+\beta)}
= \beta^{-2} z(\frac{u}{r_0 |\sin \varphi_0|})^{2(1+\beta)/\beta}$ and
$$
\tmax = |\sin \varphi_0|^{-\frac{1}{\beta}} 
\underbrace{ \beta^{-1} r_0^{-\frac{1}{\beta}}  \int_{|\sin \varphi_0|}^1  \frac{1}{u^2} \cdot 
\sqrt{ \frac{\beta^2 (r_0\, |\sin \varphi_0|)^{\frac{2(1+\beta)}{\beta}} + u^{\frac{2(1+\beta)}{\beta}} }
{1-u^2} } \, du  }_{I(\varphi_0)} .
$$
The integral $I(\varphi_0)$ tends to a positive constant 
$I_0 = \beta^{-1} r_0^{-\frac{1}{\beta} } \int_0^{\frac{\pi}{2}} (\sin \alpha)^\frac{\beta-1}{\beta}\, d\alpha$ as 
$\varphi_0 \to 0$, so the leading asymptotics of $\tmax$ is 
$|\sin \varphi_0|^{-1/\beta} I_0$.
This gives tails on the height function,
$$
\mu( (\theta,\varphi_0) : 2\tmax > t)
= 2\pi \, \mu\left( |\sin \varphi_0| < (\frac{t}{2r_0I(\varphi_0)})^{-\beta} \right) \sim 4\pi (\frac{t}{2r_0I_0})^{-\beta}.
$$
Applying the same formulas to \eqref{eq:disp}, we get
$$
\Dtheta(\varphi_0) = \frac{ 2\sgn(\varphi_0) }{|\sin \varphi_0|^{\frac{1+\beta}{\beta}} } 
\underbrace{ \frac{1}{\beta r_0^{\frac{1+\beta}{\beta}} } \int_{|\sin \varphi_0|}^1
\sqrt{ \frac{\beta^2 (r_0\, |\sin \varphi_0|)^{\frac{2(1+\beta)}{\beta}} + u^{\frac{2(1+\beta)}{\beta}} }
{1-u^2} } \, du  }_{J(\varphi_0)},
$$
and $J(\varphi_0) \to \beta^{-1} r_0^{-\frac{1+\beta}{\beta}} 
\int_0^{\frac{\pi}{2}} (\sin \alpha)^{\frac{1+\beta}{\beta}} \,d\alpha$
as $\varphi_0 \to 0$.
Therefore the reflection map $R:M^-_j \to M^+_j$  becomes
$$
R:(\theta^-,\varphi^-) \mapsto (\theta^-+2\sgn(\varphi_0) |\sin \varphi_0|^{-\frac{1+\beta}{\beta}} J(\varphi_0),\ 
-\varphi^-), \qquad \varphi_0 = - \varphi^-.
$$
Since $J(\varphi_0) \to 0$ as $\varphi_0 \to \pm \frac{\pi}{2}$ 
we get $R(\theta,\varphi^-) \to (\theta,\mp \frac{\pi}{2})$ as $\varphi^- \to \pm \frac{\pi}{2}$.
Inserting the above into \eqref{eq:kappa}, we find
\begin{eqnarray*}
 \kappa(\varphi_0) &=& - 2 \sgn(\varphi_0) \sqrt{1+\beta^{-2} r_0^{-2(1+\beta)/\beta} } 
 - \frac{1+\beta}{\beta^2}  \frac{ 2 \sgn(\varphi_0) |\sec \varphi_0| }{( r_0 | \sin \varphi_0| )^{(1+\beta)/\beta} } \times \\
 && \int_{|\sin \varphi_0|}^1 \frac{1}
  {\sqrt{u^{2(1+\beta)/\beta}+\beta^2 (r_0 |\sin \varphi_0|)^{2(1+\beta)/\beta} } } \,
  \frac{ u^{2(1+\beta)/\beta} } {\sqrt{1-u^2}}  \, du.
\end{eqnarray*}
As $\varphi_0$ increases from $0$ to $\pi/2$, $\kappa(\varphi_0)$ increases from $-\infty$ to
$-2 \sqrt{1+\beta^{-2} r_0^{-2(1+\beta)/\beta} }$, and it is smooth with a finite
limit as $\varphi_0 \to \pm\frac{\pi}{2}$, giving the H\"olderness of $\Dtheta$ away from $\varphi_0 = 0$.
Also $\kappa([-\frac{\pi}{2}, \frac{\pi}{2}]) \cap [-2,0] = \emptyset$, so that hyperbolicity is guaranteed.

The leading term of $\kappa(\varphi_0)$ is $C |\varphi_0|^{-\frac{1+2\beta}{\beta}}$ for some $C > 0$,
so, since $\kappa$ is a smooth function of $\varphi_0\neq 0$,
the leading term of $\kappa'(\varphi_0)$ in absolute value is
$$
\frac{1+2\beta}{\beta} C |\varphi_0|^{-\frac{1+3\beta}{\beta} } \leq C |\varphi_0|^{-3 \frac{1+2\beta}{\beta}}
= O(|2+\kappa(\varphi_0)|^3) \qquad \text{ as } \varphi_0 \to 0,
$$
whenever $\beta > -2/3$. Hence, for every $\beta> 0$, item 5.\ in Section~\ref{sec:condYT} holds.
By the same token, recalling that $\omega(\varphi_0) = \frac{2+\kappa(\varphi_0)}{\cos\varphi_0}$
(see item 2.\ in Section~\ref{sec:condYT}),
$$
\omega'(\varphi_0) := \frac{\kappa'(\varphi_0) + (2+\kappa(\varphi_0)) \tan \varphi_0}{\cos \varphi_0} \quad 
\text{ is bounded away from } 0,
$$
for $\varphi_0$ close to $0$. Therefore $\omega(\varphi_0)$ is monotone in one-sided 
neighborhoods of $\{\varphi_0 = 0\}$, and item 6.\ in Section~\ref{sec:condYT} holds.

\newpage
\section{Answers to the referee on revised version}

Once more I would like to express my thanks fr the careful reading of the revised version.
Most changes were straightforward but I list them lest no item gets forgotten.

\begin{description}
\item Page 1: The geodesic flows (reference added) is where $\beta = 1$ and $\beta = 2$ are both possible. In infinite horizon Lorentz gases, $\beta = 2$. I got mixed up in references here.
\item Page 1, line -10: references updated.
 \item Page 2: ``they reflect'' instead of ``it does reflect''
 \item Page 3, footnote: spelling incoming corrected.
 \item Page 3, line -5: $\theta^-$ corrected to $\theta^+$.
 \item Page 4, line 11: opening bracket inserted before ``in fact''
 \item Page 5, 1st line after item 6: ``een'' corrected to ``even''
 \item Page 5, 3rd line after ad.\ 3: Sentence rewritten.
 \item Page 6, rephrased the exposition a bit, making a clear distinction between two cases (i) and (ii). The arc $A_j$ is defined for both (but slightly different) but Figure (ii) only refers to (ii).
 \item Page 7, Figure 2: lettering corrected.
 \item Page 8, Formula (5): Missing factor $-1/\cos \varphi^+_j$ inserted. Also a superscript $+$ in one of the matrix entries.
 \item Page 8, Displayed formula: $-\kappa$ corrected to $+\kappa$.
 \item Page 9, line 13: Wrong variable $s$ replaced by $\varphi$.
 \item Page 12, Theorem 3.1.: $\mbox{supp}(v)$ contains a horn 
 replaced by $\inf_{x \in H_i}|v(x)| > 0$. I hope this explains what was meant. 
\end{description}

\end{document}